\documentclass[11pt,reqno]{amsart}

\usepackage{xcolor}
\usepackage{tikz-cd}
\usepackage[margin=1.12in]{geometry}
\usepackage{amsmath,amssymb,amsthm,mathtools}
\usepackage{enumitem}
\usepackage[colorlinks=true,linkcolor=blue,citecolor=blue,urlcolor=blue]{hyperref}
\usepackage{bm}
\usepackage{mathrsfs}

\numberwithin{equation}{section}

\theoremstyle{plain}
\newtheorem{theorem}{Theorem}[section]
\newtheorem{proposition}[theorem]{Proposition}
\newtheorem{lemma}[theorem]{Lemma}
\newtheorem{corollary}[theorem]{Corollary}
\newtheorem{question}[theorem]{Question}
\theoremstyle{definition}
\newtheorem{definition}[theorem]{Definition}

\newtheorem{introthm}{Theorem}

\theoremstyle{remark}
\newtheorem{remark}[theorem]{Remark}

\newcommand{\R}{\mathbb{R}}

\newcommand{\T}{\mathbb{T}}
\newcommand{\PP}{\mathbb{P}}
\newcommand{\ii}{\sqrt{-1}}
\newcommand{\ca}{\mathcal{C}}
\newcommand{\p}{\partial}
\newcommand{\bp}{\bar\partial}
\newcommand{\dc}{\mathrm{d}^{c}}
\newcommand{\ddc}{\sqrt{-1}\partial\bar \partial}
\DeclareMathOperator{\Real}{Re}

\title{ALH Gravitational Instantons and Rational Surfaces} 
\author{Yifan Chen}
\address{Universit\`a di Roma Tor Vergata, Dipartimento di Matematica, Via della Ricerca Scientifica 1, 00133 Roma, Italy}
\email{chen@mat.uniroma2.it}
\author{Chunhui Wei}
\address{School of Mathematics, Zhejiang University, Hangzhou 310058, People's Republic of China}
\email{chunhuiwei@zju.edu.cn}
\date{\today}

\begin{document}

\begin{abstract}
We construct complete ALH hyperk\"ahler metrics on complements of anticanonical elliptic curves with zero self-intersection in smooth rational surfaces. This answers the existence question of Filip--Tosatti~\cite{FT26}, related to the picture proposed by Haskins--Hein--Nordstr\"om~\cite{HHN15}. Conversely, we prove that every hyperk\"ahler ALH gravitational instanton admits, for each induced complex structure, a compactification by an anticanonical elliptic curve in a rational surface.
\end{abstract}

\maketitle
\tableofcontents

\section*{Introduction}

In this paper, we study asymptotically cylindrical (ACyl) Calabi-Yau manifolds. These are complete K\"ahler Ricci flat manifolds whose geometry at infinity converges exponentially to a product cylinder $(0,\infty)\times Y$. In complex dimension 2, these are special hyperk\"ahler gravitational instantons, namely ALH gravitational instantons. 

Examples are constructed by Tian--Yau~\cite{TY} on the complement of anticanonical divisors with trivial normal bundle. Building on the work of Tian--Yau~\cite{TY}, Hein~\cite{Hein12} constructed ALH Calabi--Yau metrics on complements of smooth fibers in rational elliptic surfaces and developed their deformation theory. Haskins--Hein--Nordstr\"om (HHN)~\cite[Theorem~D]{HHN15} established a general existence theorem for ACyl Calabi--Yau metrics on complements of anticanonical divisors with holomorphically torsion normal bundles. 

Then a natural question is whether we can construct ACyl Calabi--Yau metrics when the normal bundle is non-torsion. It is related to the picture proposed in~\cite[Remark~1.6]{HHN15} and stated explicitly in~\cite[Question~6.3.1]{FT26}. 
\begin{question}\label{ques:main}
     Let $E\subset \PP^2$ be a smooth cubic curve, let $S$ be the blow-up of $\PP^2$ at nine distinct points on $E$, and let $C\subset S$ be the strict transform of $E$. Does $S\setminus C$ always admit a complete Calabi-Yau metric? 
\end{question}

We first answer this question affirmatively by constructing complete ACyl Calabi--Yau metrics beyond the torsion case. 
\begin{introthm}[Theorem~\ref{thm:ma}]\label{thm:main}
Let $M$ be a compact K\"ahler manifold with a connected, reduced, smooth anticanonical divisor $D$ with numerically trivial normal bundle. Let $X:=M\setminus D$ and let $i: X\to M$ be the inclusion. Then for any K\"ahler class $\beta\in H^2(M,\mathbb{R})$, $X$ admits a complete asymptotically cylindrical Calabi-Yau metric in $i^*\beta$.
\end{introthm}
The metrics we constructed here are exponentially asymptotic to the following model. Consider a punctured disc bundle $\mathcal C=\{\xi\in L:0<|\xi|_h<1\}$, where $h$ is a flat Hermitian metric. If $t=-\log|\xi|_h$, then the flatness of $h$ gives $\ddc t=0$. Thus, for every $\sigma>0$, the form 
$$
    \omega_{\mathcal C,\sigma}=\pi^*\omega_D+\sigma\ddc t^2
$$
defines a cylindrical Ricci flat metric if we choose $\omega_D$ Ricci flat on $D$.

The proof starts with a K\"ahler metric in the prescribed class that approaches the cylindrical model exponentially. A potential linear in $t$ at infinity is used to meet the volume normalization while preserving the class and the model. The ACyl Monge--Amp\`ere theorem~\cite[Theorem~4.1]{HHN15} then gives the Ricci flat metric.

For higher dimensional case, Theorem~\ref{thm:main} also applies but not essentially giving new metrics since \cite[Theorem~A]{HHN15} proved a splitting structure theorem for any ACyl Ricci flat metric and \cite[Theorem~B]{HHN15} shows that the model space has to be a holomorphic torsion line bundle when the ACyl Calabi-Yau manifold is simply-connected and irreducible when its complex dimension $n>2$. Conversely, \cite[Theorem~C]{HHN15} then compactifies them by adding a divisor with torsion normal bundle at infinity and recovers the metric by \cite[Theorem~D]{HHN15}. 

For ALH gravitational instantons, \cite[Remark~1.6]{HHN15} asks the following sharper converse question motivated by \cite[Theorem~C]{HHN15}:
\begin{question}
    Let $X$ be a hyperk\"ahler ALH gravitational instanton with any fixed parallel complex structure. Does it admit a holomorphic compactification $X=M\setminus D$ such that $M$ is obtained from $\PP^2$ by blowing up 9 points along a smooth cubic curve?
\end{question}

The following theorem answers this question affirmatively for every complex structure in the hyperk\"ahler sphere:
\begin{introthm}[Theorem \ref{thm:canonical_compactification}]\label{thm:main2}
    Let $(X,g,I_1,I_2,I_3)$ be a hyperk\"ahler ALH gravitational instanton. For any $a=(a_1,a_2,a_3)\in S^2$, define $I_a=a_1I_1+a_2I_2+a_3I_3$. Then the complex surface $(X,I_a)$ has a smooth compactification $X=M_a\setminus E_a$ where $M_a$ is $\operatorname{Bl}_{p_1,\ldots,p_9}\PP^2$, where $E_a$ is the strict transform of a smooth cubic and the blow-up centers lie on its successive strict transforms. 
\end{introthm}

The proof follows the strategy of compactification result~\cite[Theorem~C]{HHN15}, which is sketched in Theorem~\ref{thm:global_compactification}, showing that every ACyl Calabi-Yau manifold with a unitary-flat cylindrical end in the sense of Definition~\ref{def:cplxAcyl} admits a K\"ahler compactification. Applying Theorem~\ref{thm:ma} we also show that the original metric could be recovered by the construction in Theorem~\ref{thm:main} using the uniqueness result~\cite[Theorem~E]{HHN15}. Thus this gives a compactification and reconstruction theorem of ACyl Calabi--Yau metric with unitary flat end at infinity. 

\begin{remark}
Chen--Chen~\cite[Theorems\~1.2]{ChenChenI} and \cite[Theorem~1.3]{ChenChenIII} proved that every ALH gravitational instanton admits, for \emph{at least one} induced complex structure, a compactification by a smooth fiber in a rational elliptic surface. Theorem~\ref{thm:main2} gives a rational compactification for \emph{every} induced complex structure. Furthermore, Chen--Chen~\cite[Theorem~1.5]{ChenChenIII} proved a Torelli theorem describing the global moduli space of hyperk\"ahler ALH gravitational instantons up to biholomorphism preserving its hyperk\"ahler structures. Theorem~\ref{thm:main} and Theorem~\ref{thm:main2} give a geometric realization of the ALH isometric moduli space matching the dimension counting in~\cite{Hein12}.
\end{remark}

Section~\ref{sec:model} describes the cylindrical models, and Section~\ref{sec:Acyl_mfd} proves the metric construction. Section~\ref{sec:compact} establishes holomorphic and K\"ahler compactification in the unitary flat setting and proves reconstruction of the original metric. Section~\ref{sec:ALH} applies this to ALH gravitational instantons and identifies their compactification as rational surfaces, completing the converse picture.

\smallskip
\noindent\textbf{Acknowledgements.} The authors thank professor Song Sun for suggesting this question, guidance and constant support. The authors are grateful to Hans-Joachim Hein and Song Sun for helpful feedback and suggestions on an earlier draft. We also thank Eleonora Di Nezza and Junsheng Zhang for many helpful discussions and suggestions. Part of this work was carried out during the authors' visit to Zhejiang University and we thank IASM for its hospitality. Y.C. was funded by the European Research Council (ERC) through the SiGMA project (Grant Agreement No. 101125012, PI: Eleonora Di Nezza). 

\smallskip
\noindent\textbf{Declaration of AI usage.} GPT 5.6 Sol and GPT 6 Astra are used to find the reference, polish the language and improve the exposition. Proof of rationalness in Theorem~\ref{thm:canonical_compactification} is suggested by GPT 6 Astra. The authors take full responsibility for the content, arguments, and correctness of the article.

\section{Complete asymptotically cylindrical manifold}\label{sec:model}
\subsection{Weighted cylindrical norms and ACyl manifolds}

Let $(Y,h_Y)$ be a closed Riemannian manifold and let
$$
    X_\infty=[0,\infty)\times Y,\qquad g_\infty=\sigma\,dt^2+h_Y,\qquad {\text{for some }}\sigma>0.
$$
We define the following weighted cylindrical H\"older norm on model space. 

\begin{definition}[Weighted cylindrical norms]\label{def:weighted-cylindrical-norms}
For a compactly supported smooth section $T$, define
$$
    \|T\|_{C^{k,\alpha}_\delta(X_\infty,g_\infty)}:=\|e^{\delta t}T\|_{C^{k,\alpha}(X_\infty,g_\infty)},\qquad k\geq0,\quad\delta\in\mathbb R.
$$
The weighted H\"older space $C^{k,\alpha}_\delta(X_\infty,g_\infty)$ is the completion in this norm. Compactly supported sections may be nonzero at the finite boundary, so no boundary condition is imposed there. We write
$$C^\infty_\delta(X_\infty,g_\infty) := \bigcap_{k\geq0}C^{k,\alpha}_\delta(X_\infty,g_\infty).
$$
\end{definition}

\begin{definition}[Riemannian ACyl manifolds]\label{def:riemannian-acyl}
A complete Riemannian manifold $(X,g)$ is \emph{asymptotically cylindrical}, or \emph{ACyl}, if there are a compact domain $K\subset X$ with smooth boundary, a cylinder $(X_\infty,g_\infty)$ as above, a diffeomorphism
$$
    \Phi:X_\infty\longrightarrow X\setminus\mathrm{int}K,
$$
and a constant $\delta>0$ such that
$$
    |\nabla^k_{g_\infty}(\Phi^*g-g_\infty)|\leq C_k e^{-\delta t}
$$
\end{definition}

Extending $t$ smoothly over the compact part of $X$ defines $C^{k,\alpha}_\delta(X,g)$ smooth tensors by pulling back to the model $X_\infty$ (see \cite[Definition~2.1]{HHN15}). Beside the natural exponential weight, we can also compare with other positive functions with the following notation:
\begin{definition}\label{d:C_infty}
For a smooth tensor $T$ and a positive function $f$ on an ACyl manifold, we write $$T=O'(f)$$ if, for every $k\geq0$, there is a constant $C_k$ and $C$ such that on $\{t>C\}$ we have 
$$
 |\nabla_{g_\infty}^k\Phi^*T|_{g_\infty}\leq C_k f.
$$ 
\end{definition}

\subsection{Complete K\"ahler Ricci flat cylindrical model}\label{sec:cplx-model}
In this section, we recall a Calabi-Yau cylindrical model. Let $(D, \omega_D)$ be an $(n-1)$-dimensional compact Calabi-Yau manifold with trivial canonical bundle. Fix a nowhere vanishing holomorphic volume form $\Omega_D$ on $D$ with some constant $c_D$ such that $\omega_D^{n-1}=c_D\Omega_D\wedge\bar\Omega_D.$

Let $L$ be a holomorphic unitary flat line bundle on $D$ with projection $\pi:L\to D$. Fix a unitary flat Hermitian metric $h$ on $L$, and denote the complex structure on total space of $L$ by $J_L$. Define $\rho:=|\cdot|_h$, $t:=-\log\rho$, and
$$
    \mathcal C_r=\{\rho<r\},\qquad \mathcal C_r^*=\mathcal C_r\setminus D, \qquad\mathcal C:=\mathcal C_1^*.
$$
In a unitary parallel holomorphic frame $s$, write $\xi=ws$ with $w=e^{-t-\ii\theta}$. We use the conventions $\dc_J:=\frac{i}{2}(\bar\partial_J-\partial_J), \ddc_J:=\ii\partial_J\bar\partial_J $. Denote the connection form by $\eta:=2\dc_{J_L}t=d\theta$. Fiber direction $dw/w=-dt-\ii\eta$ defines a global logarithmic one-form on $\mathcal C$, denoted by $d\xi/\xi$. We have a nowhere vanishing volume form
$$
    \Omega_\mathcal C:=\frac{d\xi}{\xi}\wedge \pi^*\Omega_D.
$$

Since $\ddc_{J_L}t=0$ and $\ddc_{J_L}t^2=dt\wedge\eta$, we obtain the K\"ahler form for any $\sigma>0$ 
$$
    \omega_{\mathcal C,\sigma}:=\pi^*\omega_D+\sigma\ddc_{J_L}t^2=\pi^*\omega_D+\sigma\,dt\wedge\eta
$$
with corresponding cylindrical metric 
$$
    g_{\mathcal C,\sigma}:=\pi^*g_D+\sigma(dt^2+\eta^2).
$$
Straightforward computation shows that $\omega_{\mathcal C,\sigma}^n=c_\sigma\,\Omega_\mathcal C\wedge\bar\Omega_\mathcal C$ for some constant $c_\sigma$. Thus $g_{\mathcal C,\sigma}$ is Ricci flat. 

For Calabi-Yau manifolds, we include the complex structure and holomorphic volume form in the asymptotic data.
\begin{definition}[ACyl Calabi-Yau manifolds]\label{def:cplxAcyl}
A complete Calabi-Yau manifold $(X,J,\omega,\Omega,g)$, is said to have a \emph{unitary-flat cylindrical end} if there exists a cylindrical model $(\mathcal C,J_L,\omega_\mathcal C,\Omega_\mathcal C,g_\mathcal C)$ as above and a diffeomorphism
$$
    \Phi:\{t\geq t_0\}\subset\mathcal C \longrightarrow X\setminus \mathrm{int}K
$$
such that, for some $\delta>0$,
\begin{equation}\label{eq:complex-acyl-weighted}
    \Phi^*J-J_L,\quad \Phi^*\omega-\omega_\mathcal C,\quad \Phi^*\Omega-\Omega_\mathcal C \in C^\infty_\delta(\{t\geq t_0\},g_\mathcal C).
\end{equation}
\end{definition}

\section{A construction of complete asymptotically cylindrical Calabi-Yau manifolds}\label{sec:Acyl_mfd}

We now move from the model geometry to the complete noncompact manifolds. Let $M$ be a compact K\"ahler manifold, and let $D\subset M$ be a smooth connected anticanonical divisor with numerically trivial normal bundle $L=N_D$.

Let $\iota:D\to M$ be the inclusion. We fix a K\"ahler class $\beta$ on $M$ and define $\beta_D:=\iota^*\beta$. By adjunction and Yau's theorem, there exists a unique Ricci flat metric $\omega_D$ on $D$ representing $\beta_D$. Since $[\omega_D]=\iota^*\beta$, the extension theorem \cite[Proposition~8.8]{GuedjZeriahi2017} gives a K\"ahler metric $\omega_M$ representing $\beta$ on $M$ such that $\iota^*\omega_M=\omega_D$, with corresponding Riemannian metric $g_M$. Let $X=M\setminus D$.

\subsection{Initial K\"ahler metric with exponential error}

Choose a defining section $S\in \Gamma(M,[D])$ up to a constant. Under $[D]\simeq K_M^{-1}$, the inverse $S^{-1}$ is a meromorphic section of $K_M$ with a simple pole along $D$, and its restriction is a nowhere-vanishing holomorphic volume form $\Omega_X$ on $X$. We fix $S$ by normalizing
$\mathrm{Res}_D\,\Omega_X = \Omega_D$, and also fix a Hermitian metric $h_M$ on $[D]\simeq K_M^{-1}$ such that its restriction to $D$ is the chosen flat Hermitian metric $h$ on $L=[D]|_D$, and $|S|_{h_M}< 1$.

In the following asymptotic estimates, norms and covariant derivatives on the model are taken with respect to $g_{\mathcal C,\sigma}$ for fixed $\sigma$.

The following proposition gives the asymptotic comparison between the complex structure of a neighborhood of the zero section in $L$ with the complex structure of a tubular neighborhood of $D\subset M$ by the normal exponential map associated with $g_M$. It is the analogue of the standard Tian--Yau asymptotic estimates in~\cite{ConlonHein15} and~\cite{HSVZ2022}. The normal exponential map $\Phi$ gives us the closeness of holomorphic volume form $\Omega_X$ and $\Omega_\ca$ as the proof of \cite[Proposition~3.4(b)]{HSVZ2022} which yields the closeness of complex structure as shown in \cite[Lemma~2.14]{CH13}. However, we present here a different proof formulated directly in terms of the complex structure itself (see also~\cite[lemma 2.6]{JZ24}).

\begin{proposition}\label{prop:Acyl_asym}
There exists an almost holomorphic diffeomorphism $\Phi$ from a neighborhood of $\mathbf{0}_L$ in $L$ onto a neighborhood of $D$ in $M$. More precisely, we have $$\Phi^*J_{M}-J_L=O'(e^{-t}).$$
\end{proposition}

\begin{proof}
For any $\xi\in L\simeq T^{1,0}M|_D/T^{1,0}D$, choose a unique tangent vector representative $v\in T_{\pi(\xi)}^{1,0}M$ in the orthogonal complement of $T_{\pi(\xi)}^{1,0}D$ with respect to $g_M$. Define the normal exponential 
$$
    \Phi(\xi)=\exp_{\pi(\xi),g_M}(2\Real v).
$$
The restriction $d\Phi|_D:TL|_{\mathbf 0_L}\simeq TD\oplus L\longrightarrow TM|_D$ is given by the above identification. Thus for $|\xi|_h<\epsilon$, with $\epsilon>0$ sufficiently small, we obtain a diffeomorphism $\Phi:\{0<|\xi|_h<\epsilon\}\longrightarrow X\setminus K$ for a compact subset $K\subset X$.

Denote $J:=\Phi^*J_M$, then the error $E:=J-J_L$ is a smooth tensor on a neighborhood of $\mathbf 0_L$ in $L$ and vanishes on the zero section $\mathbf 0_L$. 
Take a simply connected holomorphic coordinate chart $U\subset D$ with coordinates $w_2,\ldots,w_n$ and a flat unitary holomorphic frame $s$ of $L|_U$ then for any $\xi\in L|_U$, write $\xi=w_1s$, $|w_1|=e^{-t}$. 
Write $E$ in the local expansion 
$$
    E = E_\alpha^\beta\, dw^\alpha\otimes\p_{w_\beta}
$$
where $E^\beta_\alpha$ are smooth functions in a neighborhood $V$ of $U$ in $L$, and vanish on $\mathbf 0_L$, thus
$$
    |\nabla^k E_\alpha^\beta|_{g_{\mathcal C,\sigma}} = O(e^{-t}).
$$
However, by straightforward computation we have 
$$
    |\nabla^k \p_{w_1}|_{g_{\mathcal C,\sigma}} = O(e^t),\quad |\nabla^k dw_1|_{g_{\mathcal C,\sigma}} = O(e^{-t})
$$ 
and the tangential frames have bounded covariant derivatives. 
Thus we need to prove the terms $|E_\alpha^\beta|$ for tangential $\alpha\in \{2,\ldots,n, \bar 2,\ldots, \bar n\}$ and normal $\beta\in \{1,\bar 1\}$ are $O(|w_1|^2)$.

We first show that $E_\alpha^1 = O(|w_1|^2)$ for $\alpha\in \{2,\ldots, n\}$. Since $J^2 = -\mathrm{Id}$, we know that $E$ anti-commutes with $J_L$ up to $O(|w_1|^2)$, thus The first-order term of $E$ is anti-holomorphic along $D$ which gives $E_\alpha^1 = O(|w_1|^2)$. 

Next we show $E_{\alpha}^{\bar 1} = O(|w_1|^2)$. For $\alpha,\beta \in \{ 2,\dots,n\}$, we may write
$$
    J\partial_{w_\alpha} = i\partial_{w_\alpha}+(P^1_\alpha w_1+P^{\bar 1}_\alpha \bar w_1)\partial_{\bar w_1}+(Q_{\alpha}^{\bar\beta1}w_1+Q_{\alpha}^{\bar \beta\bar 1}\bar w_1)\,\partial_{\bar w_\beta}+O(|w_1|^2),
$$
where $P$ and $Q$ are smooth functions depending on $(w_2,\ldots,w_n)$, and the $O(|w_1|^2)$ term means a smooth tensor with coefficient in $O(|w_1|^2)$ when written in the ordinary smooth frame. 
Consider $J$-holomorphic vectors $Z_\alpha: = \partial_{w_\alpha}-i J\partial_{w_\alpha}$. Since $J$ is integrable, $[Z_1,Z_\alpha]\in T^{1,0}_{J}L$, so 
$$
    [Z_1,Z_\alpha]=-2iP^1_\alpha\,\partial_{\bar w_1}-2iQ_{\alpha}^{\bar\beta1}\,\partial_{\bar w_\beta}+O(|w_1|).
$$
Thus $P^1_\alpha = Q_{\alpha}^{\bar \beta1} = 0$. 

On the other hand we have 
$$
    [\p_{w_1},J\p_{{\bar w}_{\alpha}}] = P^{1}_{\bar \alpha} \partial_{ w_1} + Q_{\bar \alpha}^{\beta 1}\,\partial_{w_{\beta}}+O(|w_1|).
$$

Under the identification of $d\Phi|_D$ we can view $\p_{w_1}|_D$ as a smooth section of $T^{1,0}M|_D$ whose image in the normal quotient $L$ is the holomorphic frame $s$. 
Notice that $L$ is precisely the holomorphic normal bundle of $D$ in $M$. This implies that on $D$, the $d\bar w_\alpha$ component, for $\alpha\in\{2,\cdots, n\}$ of $\bp_{J}(\p_{w_1}|_D)$ is tangential to $D$. 
More precisely, let $\pi_L: T^{1,0}M|_D\to L$ be the quotient map, we have 
$$
    \pi_L([\p_{w_1},J\p_{\bar w_{\alpha}}]|_D)
    =2i\,\pi_L((\bp_J(\p_{w_1}|_D))(\p_{\bar w_{\alpha}})) = 2i\bp_L(\pi_L(\p_{w_1}|_D))(\p_{\bar w_{\alpha}}) =0 
$$
Thus $P^{\bar 1}_\alpha = 0$, so $E^{\bar 1}_\alpha = O(|w_1|^2)$. The same estimates hold for $\alpha\in \{\bar 2,\ldots, \bar n\}$. Combining this with the estimates of the smooth frame we get
$$
    \Phi^*J_M-J_L=O'(e^{-t}).
$$
\end{proof}

The closeness of holomorphic volume form follows from the proof in \cite[Proposition~3.4(b)]{HSVZ2022} and also could be seen as a consequence of the closeness of complex structure together with the normalization of $\Omega_X$ and $\Omega_{\mathcal C}$ by fixing the same residue along $D$.

\begin{proposition}\cite[Proposition~3.4(b)]{HSVZ2022}\label{prop:volume-form-asymptotics}
Under the normal exponential map $\Phi$ of Proposition~\ref{prop:Acyl_asym}, the holomorphic volume forms satisfy
$$
    \Phi^*\Omega_X - \Omega_{\mathcal C}=O'(e^{-t}).
$$
\end{proposition}

\begin{lemma}\label{cor:ddc-decay}
Let $t_X:=-\log|S|_{h_M}$ on $X$. It is close to the model potential exponentially, i.e. $\Phi^*t_X-t = O'(e^{-t})$. Moreover, $\Phi^*(\ddc_{J_X}t_X) = O'(e^{-t})$.
\end{lemma}
\begin{proof}
Since $d\Phi$ induces this identification along the zero section, we have 
$$
    \Phi^*|S|_{h_M}^2= |w_1|^2(1+O'(e^{-t}))
$$
Hence 
$$
    \Phi^*t_X-t = -\frac12\log \bigl( \Phi^*|S|_{h_M}^2\bigr) + \frac12\log|w_1|^2 = O'(e^{-t}).
$$
Write $\ii\Theta_{h_M}$ for the Chern curvature form of $h_M$, so that $\ddc_{J_X}t_X=\tfrac12\ii\Theta_{h_M}|_X$. The restriction $h_M|_D=h$ is unitary flat, so $\iota^*\Theta_{h_M}=0$. Thus $\Phi^*\Theta_{h_M}$ is a smooth two-form whose purely tangential restriction to the zero section vanishes. Any component containing a normal 1-form is already $O'(e^{-t})$, while the tangential components have coefficients vanishing on the zero section. It follows that $\Phi^*(\ddc_{J_X}t_X) = \tfrac12\ii\Phi^*\Theta_{h_M}=O'(e^{-t})$.
\end{proof}

For any $\sigma>0$, we choose a gluing convex smooth function $q$ on $\mathbb{R}^+$ such that $q = 0$ when $s \leq S$ and $q(s) = \sigma s^2 + c_1s+ c_2$ when $s \geq S+1$. By pushing $S$ sufficiently large and suitable choice of $c_1,c_2$, we have
$$
    \omega_{X,\sigma}:=\omega_M|_X+\ddc_{J_X}q(t_X)
$$
defines a K\"ahler form on $X$, with associated metric $g_{X,\sigma}$. 
Combining Proposition~\ref{prop:Acyl_asym}, Lemma~\ref{cor:ddc-decay}, and $\Phi^*\omega_M-\pi^*\omega_D = O'(e^{-t})$ we have the closeness of initial K\"ahler metrics:

\begin{corollary}\label{cor:metric_asym}
The complete K\"ahler manifold $(X, g_{X,\sigma})$ is asymptotically cylindrical. More precisely, we have
$$
    \Phi^*\omega_{X,\sigma}-\omega_{\mathcal C,\sigma}=O'(te^{-t}).
$$
Let $c_\sigma$ be the model normalization from Section~\ref{sec:cplx-model}. The Ricci potential of $\omega_{X,\sigma}$ satisfies
$$
    f_{X,\sigma}:=\log\frac{c_\sigma\,\Omega_X\wedge\bar\Omega_X}{\omega_{X,\sigma}^n}=O'(te^{-t}).
$$
\end{corollary}

\subsection{Modification by a linear potential at infinity}\label{sec:perturb_existence}
In this subsection we omit $\Phi^*$ and fix $\sigma>0$. We write $\omega_\mathcal C$ for its cylindrical model and let $c$ denote $c_\sigma$. We will modify the potential by a term linear in $t_X$ at infinity.
We start with $\omega_X: = \omega_{X,\sigma}$ and $g_X:=g_{X,\sigma}$. 
Since the form $\omega_X+2A\,\ddc_{J_X} t_X$ is not necessarily K\"ahler for arbitrary $A$, we use a slowly varying cutoff. Let $\chi_{R,T}:\mathbb R\to\mathbb R$ be a smooth convex function satisfying
\begin{align*}
    \chi_{R,T}(t)=0 \text{ for } t\leq R,&\qquad \chi_{R,T}(t)=t-c_{R,T}\text{ for } t\geq R+T,\\
    0\leq \chi_{R,T}'\leq 1,&\qquad 0\leq \chi_{R,T}''\leq \frac{C}{T},
\end{align*}
where $c_{R,T}$ is a constant depending on $R,T$ and $C$ is independent of $R,T\geq 1$. For $A\in\mathbb R$, define
$$
    \omega_{A,R,T}:=\omega_X+2A\,\ddc_{J_X}\chi_{R,T}(t_X).
$$

\begin{lemma}\label{lem:kahler-cutoff-perturbation}
For every fixed $A\in\mathbb R$, there exist $R,T$ large enough such that $\omega_{A,R,T}$ is a K\"ahler form on $X$. Moreover, $\omega_{A,R,T}$ is asymptotic to the same cylindrical model $\omega_\mathcal C$.
\end{lemma}
\begin{proof}
On the transition region $R\leq t_X\leq R+T$,
$$
    \ddc_{J_X}\chi_{R,T}(t_X)=\chi_{R,T}'(t_X)\,\ddc_{J_X} t_X+\chi_{R,T}''(t_X)\,dt_X\wedge\dc_{J_X} t_X .
$$
Hence
$$
    |\ddc_{J_X}\chi_{R,T}(t_X)|_{g_X}\leq C(e^{-R}+T^{-1})
$$
on the transition region. On $\{t_X\leq R\}$, the perturbation vanishes. On the tail $\{t_X\geq R+T\}$ we have $\ddc_{J_X} t_X=O'(e^{-t})$. Therefore, for the fixed value of $A$, by choosing first $R$ large and then $T$ large, we can make $2A\,\ddc_{J_X}\chi_{R,T}(t_X)$ arbitrarily small in $C^0$-norm with respect to $g_X$. Since $\omega_X$ is K\"ahler, this implies that $\omega_{A,R,T}$ is K\"ahler. Finally, on the tail, $\omega_{A,R,T}-\omega_X=2A\,\ddc_{J_X} t_X=O'(e^{-t})$, while $\omega_X-\omega_\mathcal C=O'(te^{-t})$. Thus $\omega_{A,R,T}$ has the same cylindrical model $\omega_\mathcal C$.
\end{proof}

\begin{lemma}\label{lem:integral-condition}
There exists a unique $A\in\mathbb R$ such that, for any choice of $R,T$ for which $\omega_{A,R,T}$ is K\"ahler,
$$
    \int_X(\omega_{A,R,T}^n-c\,\Omega_X\wedge\bar\Omega_X)=0 .
$$
with $c$ normalized by $\omega_\mathcal C^n = c\,\Omega_\mathcal{C}\wedge\bar\Omega_\mathcal{C}$.
\end{lemma}
\begin{proof}
First, the integral $\int_X(\omega_X^n-c\,\Omega_X\wedge\bar\Omega_X)$ is finite since the difference $\omega_X^n-c\,\Omega_X\wedge\bar\Omega_X$ decays exponentially on the cylindrical end. 

We use $\varphi_{A,R,T}:=2A\chi_{R,T}(t_X)$ to perturb the integral of $\omega_{A,R,T}=\omega_X+\ddc_{J_X} \varphi_{A,R,T}$. Extend the model coordinate $t$ smoothly over $X$ and let $X_s:=\{t\leq s\}$, where $s$ is large enough such that $t_X>R+T$ on $\partial X_s$. Then we have $\chi_{R,T}'(t_X)=1$, $\dc\varphi_{A,R,T}=2A\dc t_X$, $\omega_{A,R,T}=\omega_X+2A\ddc t_X$ on $\partial X_s$. Stokes' theorem gives
\begin{align*}
    \int_{X_s}(\omega_{A,R,T}^n-\omega_X^n)
    & = \int_{\partial X_s}\dc\varphi_{A,R,T}\wedge\sum_{j=0}^{n-1}\omega_{A,R,T}^{j}\wedge\omega_X^{n-1-j}\\
    & = \int_{\partial X_s}2A\dc t_X\wedge\sum_{j=0}^{n-1}(\omega_X+2A\ddc t_X)^{j}\wedge\omega_X^{n-1-j}\\
    & = \int_{\partial X_s}2A\dc_{J_L}t\wedge\sum_{j=0}^{n-1}(\omega_\mathcal C+2A\ddc_{J_L}t)^{j}\wedge\omega_\mathcal C^{n-1-j}+ O(e^{-\delta s})
\end{align*}
for any fixed $0<\delta<1$. Let $Y$ be a model cross-section, oriented as the boundary of $X_s$, and let
$$
    I_Y:=\int_Y\eta\wedge\pi^*\omega_D^{n-1}>0.
$$
Since $\ddc_{J_L}t=0$ and $2\dc_{J_L}t=\eta$, letting $s\to\infty$ gives
$$
    \int_X(\omega_{A,R,T}^n-\omega_X^n) = 2An\int_Y\dc_{J_L}t\wedge\omega_\mathcal C^{n-1} = An I_Y.
$$
The restriction of $\omega_\mathcal C$ to $Y$ is $\pi^*\omega_D$, so $I_Y$ is independent of $\sigma$. Equivalently, for the cross-section
metric $g_{Y,\sigma}:=\sigma\eta^2+\pi^*g_D$, one has $I_Y=(n-1)!\,\sigma^{-1/2}\operatorname{Vol}(Y,g_{Y,\sigma})$. Thus
$$
    \int_X(\omega_{A,R,T}^n-c\,\Omega_X\wedge\bar\Omega_X) = \int_X(\omega_X^n-c\,\Omega_X \wedge \bar\Omega_X)+AnI_Y
$$
is a nonconstant affine function of $A$, independent of $R,T$. This proves the existence and uniqueness of the required $A$.
\end{proof}

\begin{theorem}\label{thm:ma}
With the above notation, $X=M\setminus D$ admits a complete asymptotically cylindrical Calabi-Yau metric $\omega$ in $[\omega_M|_X]$.
\end{theorem}
\begin{proof}
Fix $\sigma,A,R,T$ as before. By Corollary~\ref{cor:metric_asym} and Lemma~\ref{lem:kahler-cutoff-perturbation}, we have
$$
    \omega_{A,R,T}-\omega_\mathcal C=O'(te^{-t}),\quad \Omega_X-\Omega_{\mathcal C}=O'(e^{-t}).
$$ 
Thus the Ricci potential
$$
    f_{A,R,T} :=  \log \frac{c\Omega_X\wedge\bar \Omega_X}{\omega_{A,R,T}^n}=O'(te^{-t}).
$$
Moreover, Lemma~\ref{lem:integral-condition} gives
$$
    \int_X(e^{f_{A,R,T}}-1) \omega_{A,R,T}^n = \int_X (c\,\Omega_X\wedge\bar\Omega_X-\omega_{A,R,T}^n) = 0.
$$
The ACyl K\"ahler Ricci flat theorem~\cite[Theorem~4.1]{HHN15} therefore gives a solution $u\in C^\infty_\delta(X)\cap \mathrm{PSH}(X,\omega_{A,R,T})$ for some $\delta>0$ such that 
$$
    (\omega_{A,R,T}+\ddc_{J_X}u)^n = e^{f_{A,R,T}} \omega_{A,R,T}^n = c\,\Omega_X\wedge\bar\Omega_X.
$$
The decay of $u\in C^\infty_\delta(X)$, Lemma~\ref{lem:kahler-cutoff-perturbation}, and Proposition~\ref{prop:Acyl_asym} imply that
$$
    (\omega_{A,R,T}+\ddc_{J_X}u)-\omega_\mathcal C=O'(te^{-\delta' t}),
$$
for any $0< \delta' < \min\{1,\delta\}$. Then $\omega: =\omega_{A,R,T}+\ddc_{J_X}u$ is the desired asymptotically cylindrical K\"ahler Ricci flat metric.
\end{proof}

\section{Compactification of complete asymptotically cylindrical Calabi-Yau manifolds}\label{sec:compact}

In this section, we show the following classification result: 
\begin{theorem}\label{thm:global_compactification}
 Every ACyl Calabi--Yau manifold $(X,J_X)$ with a unitary flat cylindrical end is biholomorphic to $M\setminus D$, where $M$ is a compact K\"ahler manifold and $D$ is a smooth anticanonical divisor with flat normal bundle. The K\"ahler Ricci flat metric $\omega$ could be recovered from the construction of Theorem~\ref{thm:ma} applied to $(M, D)$.
\end{theorem}

The holomorphic compactification follows from the unitary flat version of HHN \cite[Theorem~3.1]{HHN15}, which extends the complex structure $J_X$ across the zero section of $L$ by fixing the gauge of holomorphicity along fiber.
\begin{theorem}\label{thm:extension} 
Let $J$ be a smooth integrable complex structure on $\mathcal C_{r_0}^*$ such that, for some $\tau>0$, $J-J_L=O'(e^{-\tau t})$. Then, for some $0<r<r_0$, there are a punctured neighborhood $U\subset\mathcal C_{r_0}^*$ of the zero section and a diffeomorphism $F:\mathcal C_r^*\to U$ such that $\widetilde J:=F^*J$ extends smoothly to $\mathcal C_r$. Moreover, the fiber discs are $\widetilde J$-holomorphic, and
$$\widetilde J|_D=J_L|_D,
 \qquad N_{D/(\mathcal C_r,\widetilde J)}\simeq L.$$
\end{theorem}

\begin{proof}[Sketch of the proof]
The proof follows almost verbatim from HHN~\cite[Section~3.2]{HHN15} so we only give a sketch of proof here. Choose unitary parallel frames $s_i$ of $L$. If $s_j=g_{ij}s_i$, then the corresponding fiber coordinates satisfy
$$w_j=g_{ij}^{-1}w_i,\qquad g_{ij}\in U(1)\text{ constant}.$$
Thus the local product models differ only by constant rotations of the fibers. The radial coordinate, the cylindrical norms, and the flat horizontal distribution are global.

First, in order to remove the possible pole of normal-to-horizontal term in $J-J_X$, we fix a gauge of diffeomorphism to its image $F:\mathcal{C}_r^*\to \mathcal{C}_{r_0}^*$ such that the $F^*J$ equal to $J_L$ on vertical tangent vectors. This is done by gluing the fiberwise holomorphic disc diffeomorphism $f_x: C_{x,r}\to L$ smoothly, where $C_{x,r}: = L_x\cap\mathcal{C}^*_r$ be the holomorphic disc over $x$. Here we try to find such $f_x$ as an exponential map
$$
    f_{x,v}(p) =\exp^{g_\mathcal{C}}_{p} v(p), \text{ for }v\in C^{k,\alpha}_\delta(C_{x,r},TL|_{C_{x,r}}) 
$$
satisfies the nonlinear equation 
$$
    df_{x,v}(\partial t) +J(df_{x,v}(\partial_\theta)) = 0.
$$ 

The linearization of this nonlinear equation is the standard Cauchy-Riemann operator
$$
    \nabla_{\partial_t} v +J_L \nabla_{\partial_\theta} v = 0
$$
on punctured disc together with a small perturbation operator $\mathcal{U}_x$ varying smoothly with $x$. By choosing base points $x_i$ in a given chart we can choose the fiberwise bounded right inverse $\mathcal{R}_x$ of $\bar \partial$ also smoothly varying with $x$. Then by standard fixed point argument, we can find $v$ such that $f_{x,v}(C_{x,r})$ smoothly varies with $x$ and solves above non-linear equation. Then, they assemble to a diffeomorphism $F$ exponentially close to the identity that we need.

Second, the vanishing of the Nijenhuis tensor gives the same fiberwise elliptic equation as in HHN. Hence, the bootstrap in HHN \cite[Section~3.2, equations~(3.10)--(3.12)]{HHN15} gives a $C^{1,\alpha}$ extension for $0<\alpha<1$. Their decay bootstrap and Newlander--Nirenberg argument apply in each flat chart, and the subsequent Cauchy expansion along the holomorphic fiber discs gives smoothness across $D$. These are local tensor statements, so no global fiber coordinate is required.

Finally, HHN's normal-bundle argument shows that $\nu_i=[\partial_{w_i}]$ is a holomorphic local normal section. On overlaps $\nu_j=g_{ij}\nu_i$, which identifies the normal bundle with $L$.
\end{proof}

The same compatibility with unitary flat transitions applies to the first proof of HHN \cite[Theorem~3.2]{HHN15}.

\begin{theorem}\label{thm:kahler_compactification}
In the setting of Theorem~\ref{thm:extension}, suppose that there is a
$J$-K\"ahler form $\omega$ on $\mathcal C_{r_0}^*$ such that, for some $\tau>0$ and every $m\geq0$,
$$
    |\nabla^m(\omega-\omega_\mathcal C)|_{g_\mathcal C} =O(e^{-\tau t}), \qquad \omega_\mathcal C=\pi^*\omega_D+\sigma\ddc_{J_L}t^2.
$$
After shrinking $r$, $(\mathcal C_r,F^*J)$ admits a smooth K\"ahler form $\widehat\omega$ agreeing with $F^*\omega$ on $\{r/2<\rho<r\}$ and $\widehat\omega|_D =\omega_D$.
\end{theorem}

\begin{proof}[Sketch of the proof]
Let $\widetilde\omega=F^*\omega$, $\widetilde J=F^*J$, and $s=t+\log r$. Choose constants $c_1,c_2,C$, with $c_1$ sufficiently negative, $c_2$ suitably large, $C>0$, so that a smooth strictly convex function $\phi(s)$ equals $P(s)=\sigma s^2+c_1s+c_2$ for $s\leq3$ and $Ce^{-2s}$ for $s\geq5$. After shrinking $r$, write
$$
    \widetilde\omega=\ddc_{\widetilde J}P(s)+\pi^*\omega_D+\psi_r.
$$
The closed error $\psi_r$ is exponentially decaying in cylindrical norms. Following \cite[Theorem~3.2, equations~(3.18)--(3.20)]{HHN15}, radial integration gives $\psi_r=d(\beta+\bar\beta)$ with $\beta$ of type $(0,1)$ for $\widetilde J$. Solve $\bar\partial f_x=\beta$ along each holomorphic fiber and assemble to a global function $f$. Let
$$
    u=-2\operatorname{Im}f,\qquad \kappa=\beta-\bar\partial f,\qquad \psi_r+\ddc_{\widetilde J}u=d(\kappa+\bar\kappa).
$$
Here $f$ decays exponentially. The form $\kappa$ vanishes on $\ker d\pi$, and \cite[(3.19)]{HHN15} implies that it extends smoothly across $D$ with $\kappa|_D=0$. Thus, it has decay order at least 1.

Choose $\chi(s)=0$ for $s\leq1$ and $\chi(s)=1$ for $s\geq2$, and denote $\kappa_\chi=\beta-\bar\partial(\chi f)$. Define
\begin{align*}
    \widehat\omega &:=\pi^*\omega_D+d(\kappa_\chi+\bar\kappa_\chi)+\ddc_{\widetilde J}\phi(s),\\
    \widehat\omega-\widetilde\omega &=\ddc_{\widetilde J}\bigl(\chi u+\phi(s)-P(s)\bigr).
\end{align*}
Thus $\widehat\omega$ is closed and of type $(1,1)$, and equals $\widetilde\omega$ for $s\leq1$, in particular on $r/2<\rho<r$. Near $D$, $\kappa_\chi=\kappa$ and $\phi=C\rho^2/r^2$. Considering the type of $\widehat\omega$ we know that $d(\kappa+\bar\kappa) + (\pi^*\omega_D)^{0,2} + (\pi^*\omega_D)^{2,0}$ extends smoothly and vanishes when restricted to $TD$. For the inclusion $\iota_D:D\hookrightarrow\mathcal C_r$, we have
$$
    \widehat\omega|_D:=\iota_D^*\widehat\omega=\omega_D,
$$
since $\iota_D^*(d(\kappa+\bar\kappa))=0$ and $\iota_D^*\ddc_{\widetilde J}\rho^2=0$. Finally, strict convexity and smallness give positivity on the bounded transition region. Near $D$, after the fiber rescaling $w=r\zeta$,  make the corrected form a small perturbation of $(\pi^*\omega_D)^{1,1}+C\ddc_{J_L}|\zeta|^2$. Hence $\widehat\omega$ is K\"ahler after shrinking $r$.
\end{proof}

Now we are ready to present the proof of Theorem~\ref{thm:global_compactification}.
\begin{proof}[Proof of Theorem~\ref{thm:global_compactification}]
We use the gluing argument of HHN \cite[Section~3.1]{HHN15}, with the unitary-flat local results proved above. Let $(J,\omega,\Omega,g)$ be the ACyl Calabi-Yau structure on $X$, and choose its end identification $\Phi:\mathcal C_{r_0}^{*}\longrightarrow X\setminus K$ for a compact subset $K\subset X$. 

The weighted condition \eqref{eq:complex-acyl-weighted} gives, for some $\tau>0$ and every $m\geq0$,
    \begin{equation}\label{eq:global-acyl-estimates}
      \left|\nabla^m(\Phi^*J-J_L)\right|_{g_\mathcal C}+\left|\nabla^m(\Phi^*\omega-\omega_\mathcal C)\right|_{g_\mathcal C} +\left|\nabla^m(\Phi^*\Omega-\Omega_\mathcal C)\right|_{g_\mathcal C}=O(e^{-\tau t}).
    \end{equation}

Apply Theorem~\ref{thm:extension} to $\Phi^*J$. After shrinking the radius, it gives a diffeomorphism $F$ asymptotic to the identity for which
    $$
        \widetilde J:=F^*\Phi^*J
    $$
    extends smoothly and integrably over the zero section $D\subset \mathcal C_r$, and
    $$
        N_{D/(\mathcal C_r,\widetilde J)}\simeq L.
    $$
    The estimates in \eqref{eq:global-acyl-estimates} are preserved, with a possibly smaller rate, by this asymptotic change of coordinates. Restrict the target of $F$ to a standard punctured disc bundle $\mathcal C_r^*$. Theorem~\ref{thm:kahler_compactification} gives a smooth $\widetilde J$-K\"ahler form $\widehat\omega$ on $\mathcal C_r$ such that
    $$
        \widehat\omega=F^*\Phi^*\omega
        \qquad\text{on }\{r/2<\rho<r\}.
    $$

    On the punctured cap, denote
    $$
        \Psi=\Phi\circ F.
    $$
    Use the open outer piece
    $$
        X_{\mathrm{out}}=X\setminus\Psi\bigl(\{0<\rho\leq r/2\}\bigr).
    $$
    Since $\widetilde J=F^*\Phi^*J$, $\Psi$ is biholomorphic. Glue $X_{\mathrm{out}}$ and $\mathcal C_r$ along $\{r/2<\rho<r\}$ using $\Psi$. The resulting complex manifold $M$ is smooth and compact, and $M\setminus D\simeq X$.

    The equality on the outer annulus also gives a global K\"ahler form: define $\omega_M$ to be $\omega$ on $X_{\mathrm{out}}$ and $\widehat\omega$ on $\mathcal C_r$. These forms agree under $\Psi$ on the annulus, so they glue to a smooth closed positive $(1,1)$-form on $M$.

    Next we identify the canonical bundle. Fix a local chart of $D$, denote it by $U$. Then $U\times \Delta_w$ be a chart in $L$. Let $u=(u_1,u_2,\cdots, u_{n-1})$ be the holomorphic coordinate on $D$ extended to local holomorphic chart on $M$ and pull back to $L$. Let $z$ be a local $\widetilde J$-holomorphic defining function for $D$. Along each holomorphic fiber its expansion is
    $$
        z=w\bigl(a+O(w)\bigr),\qquad a\ne0.
    $$
    Since $w\Omega_\ca=dw\wedge\pi^*\Omega_D$ and $\widetilde\Omega-\Omega_\ca=O'(\rho^{\tau'})$, we have $z\widetilde\Omega-adw\wedge\pi^*\Omega_D=O'(\rho^{1+\tau''})$ for some $\tau''>0$. Since $adw\wedge\pi^*\Omega_D$ is nowhere-vanishing on $D$, locally we have the expansion
    $$
    \widetilde\Omega:=F^*\Phi^*\Omega = b(w,u)\frac{dz}{z}\wedge du_1\wedge\cdots\wedge du_{n-1},\quad b(0,u)\neq 0
    $$
    on the punctured cap.  It implies that $z\widetilde\Omega$ is bounded and holomorphic on the punctured chart, and extends holomorphically across $D$ by removable singularities. These local extensions, together with $\Omega$ on $X$, define a nowhere-vanishing section $\Omega_M\in H^0(M,K_M(D))$.
    Finally, gluing preserves a neighborhood of the zero section, so $N_{D/M}\simeq L$.

    We now prove the reconstruction part. 
    
    The proof of Theorem~\ref{thm:kahler_compactification} gives $\widehat\omega-\widetilde\omega=\ddc_{\widetilde J}v$, where $v=\chi u+\phi-P$ vanishes on the gluing annulus. Hence $v\circ\Psi^{-1}$ extends by zero to a smooth function $V$ on $X$. Denoting $\beta=[\omega_M]$, we obtain
    $$
       \omega_M|_X-\omega=\ddc_J V,\qquad i^*\beta=[\omega],\qquad \iota_D^*\omega_M=\omega_D,
    $$
    where $i:X\hookrightarrow M$ and $\iota_D:D\hookrightarrow M$ are the inclusions. In particular, $\beta$ is a K\"ahler class extending $[\omega]$ with the prescribed restriction to $D$.
    
    Apply Theorem~\ref{thm:ma} to $(M,D,\beta)$ with the original $\sigma$ and flat metric $h$ on $L$, and take $\Omega_X=\Omega$ using the meromorphic extension above. Choose the normal-bundle identification $\tilde \Psi: N_{D/M}\to M$ in Section~\ref{sec:Acyl_mfd} and $\Sigma := \Psi^{-1} \circ \tilde \Psi$. We know that $\Sigma-\mathrm{Id}$ is smooth. Moreover, writing $\Psi^{-1}$ and $ \tilde \Psi$ in local coordinates we have $\Sigma$ is $O(\rho)$ horizontally and $O(\rho^2)$ vertically in smooth coordinate, hence by $O'(e^{-t})$ in cylindrical coordinates. Thus the reconstructed form $\omega_{0}$ and $\omega$ have exponentially decaying difference in the same end coordinates, and
    $$
       [\omega_{0}]=i^*\beta=[\omega],\qquad \omega_{0}^n=c_\sigma\Omega\wedge\bar\Omega=\omega^n.
    $$
    
    The last equality holds because $\omega$ and $\Omega$ are parallel and their volume ratio is fixed by the cylindrical limit. Since $D$ is connected, $X$ has one end. The uniqueness theorem~\cite[Theorem~E]{HHN15} therefore gives $\omega_{0}=\omega$, proving reconstruction.
\end{proof}

\section{Compactification of an ALH gravitational instanton}\label{sec:ALH}
We first recall the model of ALH gravitational instanton and describe a holomorphic compactification.
\subsection{Coordinate on the model}\label{sec:coordinate}
\begin{definition}\label{d:ALH_model_Rm}
    Let $X_\infty=\R^3/\Lambda\times (R,+\infty)\simeq\T^3\times (R,+\infty)$ and let $r$ be the coordinate on $(R,+\infty)$. Let $g_\infty=g_{\T^3}+dr^2$ be a flat metric on $X_\infty$.
\end{definition}

We start with arbitrary splitting flat K\"ahler metric structure on $(\mathbb{R}^3/\Lambda^3\times\mathbb{R}, g_0, J_0)$. It uniquely lifts to $(\mathbb{R}^3\times\mathbb{R}, g, J)$. Take orthonormal coordinate $(x,y,\vartheta,r)$ we have standard flat metric $dx^2+dy^2+d\vartheta^2 + dr^2 $ and complex structure $Jd\vartheta = -dr, J dx = -dy$ on $\mathbb{R}^3/\Lambda^3\times\mathbb{R}$. Let $\Lambda^3 = \mathrm{Span}_{\mathbb{Z}}\{e_1,e_2,e_3\} $ with 
$$e_i = (a_i,b_i,c_i,0)\in \mathbb{R}^4 \text{ or } = (a_i+\sqrt{-1}b_i,c_i)\in \mathbb{C}^2.$$
By scaling the metric we can assume $e_1 = (a,b,1,0)$.

Take $A$ to be $\begin{pmatrix}
    1&-(a+\ii b)\\0&1
\end{pmatrix}$. we have following biholomorphism descending to the quotient:
$$
    A: \mathbb C^2/\Lambda^3\to \mathbb C^2/ (A\cdot\Lambda^3)
$$
$A$ maps $e_1$ to $(0,0,1,0)$, denote the new lattice by 
$$
    e_i':= e_i\cdot A=(a_i', b_i',c'_i,0):=(a_i-ac_i,b_i-bc_i,c_i,0).
$$

Denote the new lattice by $\Lambda' = \mathrm{Span}_{\mathbb{Z}}\{e'_1,e'_2,e'_3\} = \mathbb{Z}(0,0,1,0)+\Lambda^2$, where $\Lambda^2: = \mathrm{Span}_{\mathbb{Z}}\{e'_2,e'_3\}$.

Define a $U(1)$ representation $\chi$ of $\Lambda^2$ by $\chi(e_i') = \exp(2\pi\ii c_i)$. It induces an action of $\Lambda^2$ on $\mathbb C\times\mathbb{C}^*$. For any $e\in \Lambda^2$
$$
    e\cdot(z,w) =  (z+e, \chi(e) w).
$$
and $(\mathbb{C}\times\mathbb{C}^*)/\Lambda^2$ is the total space of a flat line bundle minus the zero section over $\mathbb{T}^2$, denoted by $L^\times$. The exponential map 
\begin{align*}
    \Phi: \mathbb C_{z,u}^2 / \mathbb{Z}(0,1)&\to \mathbb{C}_z\times\mathbb{C}_w^*\\
    (z,u) &\mapsto (z,\exp(2\pi \ii u)).
\end{align*} 
induces the following equivariant biholomorphism between $\mathbb C^2/\Lambda' \to L^\times$.

Since $u=\vartheta+\ii r$ and $w=\exp(2\pi\ii u)$, the logarithmic radius is $t=-\log|w|=2\pi r$. Thus, $|w|=\exp(-2\pi r)\to0$ as $r\to\infty$. Thus the biholomorphism gives a compactification of $(\mathbb{R}^3/\Lambda^3\times\mathbb{R}, g_0, J_0)$ at infinity by adding a $\mathbb{T}^2$.

View the flat metric $g_{\infty}$ on $(\mathbb{C}\times\mathbb{C},(z,u))$ under $(x,y,\vartheta,t)\mapsto (x+iy,\vartheta+i t)$. It induces a K\"ahler Ricci flat metric on $(\mathbb{C}\times\mathbb{C}^*, (z,w))/\Lambda^2$ given by
$$
    \omega = \ddc (2(1-a^2-b^2)t^2+|z+2i(a+ib)t|^2)
$$

As a summary, we have 
\begin{proposition}\label{prop:ALH_model_cplx}
    For any $J_\infty$ be a parallel complex structure on $(X_\infty, g_\infty)$, there exists a unitary flat holomorphic line bundle $L$ over an elliptic curve, such that there is a biholomorphism
    $$
        (X_\infty,J_\infty)\simeq \mathcal{C},
    $$
    with $\mathcal{C}$ the punctured disc bundle defined in Section~\ref{sec:cplx-model}.
\end{proposition}

\begin{remark}\label{rmk:split-metric}
    In this coordinate construction, we can choose a complex structure such that $e_1 = (0,0,1,0)$. Equivalently, we have a complex structure such that the cylindrical metric splits as in Section~\ref{sec:cplx-model}.
\end{remark}

\subsection{Existence of compactification}
The hyperk\"ahler ALH gravitational instanton is defined as follows.
\begin{definition}\label{d:ALH_hyperkahler}
A complete connected hyperk\"ahler $4$-manifold $(X^4,g,I_1,I_2,I_3)$ is called a \emph{hyperk\"ahler ALH gravitational instanton of order $\tau>0$} if there exist a flat ALH model $(X_\infty,g_\infty)$, a parallel hyperk\"ahler triple $(I_{\infty,1},I_{\infty,2},I_{\infty,3})$ compatible with $g_\infty$, a bounded domain $K\subset X$, and a diffeomorphism
$$
    \Phi:X_\infty\longrightarrow X\setminus K
$$
such that
$$
    \Phi^*g-g_\infty,\Phi^*I_k-I_{\infty,k} =  O'(e^{-\tau r}).
$$
\end{definition}

\begin{theorem}\label{thm:canonical_compactification}
    Let $(X,g,I_1,I_2,I_3)$ be a hyperk\"ahler ALH gravitational instanton. For any $a=(a_1,a_2,a_3)\in S^2$, define $I_a=a_1I_1+a_2I_2+a_3I_3$, the complex surface $(X,I_a)$ has a smooth compactification
    $$
        X=M_a\setminus E_a
    $$
by an elliptic curve in $|K_{M_a}^{-1}|$ with unitary flat normal bundle.

Moreover, $M_a$ is a smooth rational surface and admits a presentation
$$
    M_a\cong \operatorname{Bl}_{p_1,\ldots,p_9}\PP^2,
$$
where $E_a$ is the proper transform of a smooth plane cubic and the blow-up centers lie on its successive proper transforms. 
\end{theorem}
\begin{proof}
Fix $a\in S^2$. Proposition~\ref{prop:ALH_model_cplx} identifies the complex end model with a punctured unitary flat line bundle $L_a\to E_a$ over an elliptic curve. But the metric $g_\infty$ is not the split metric on $L_a^\times$ but twisted by a mixed class. The hyperk\"ahler asymptotics give exponential convergence, with all derivatives, of $I_a$ and its parallel holomorphic two-form to the corresponding model tensors. Since $g_\infty$ and $g_{\mathcal{C}}$ both have constant coefficients in logarithmic coordinates, these estimates also hold in the cylindrical norms under $g_{\mathcal{C}}$. The holomorphic compactification and canonical-bundle arguments in the proof of Theorem~\ref{thm:global_compactification}, using Theorem~\ref{thm:extension}, give a smooth compact complex surface $M_a$ with
$$
X\simeq M_a\setminus E_a,\qquad -K_{M_a}\simeq E_a,\qquad N_{E_a/M_a}\simeq L_a.
$$
Here only the complex-structure and volume-form estimates are needed. 

Since $-K_{M_a}$ is effective, we have $\kappa(M_a) = -\infty$. By Enriques--Kodaira classification~\cite[Theorem~VI.1.1]{BHPV04}, if $M_a$ is non-K\"ahler surface with $\kappa=-\infty$ then it is of class VII with $b_1 = 1, h^{0,2}=0, h^{0,1}= 1$ thus $\chi(\mathcal{O}_{M_a})=0$. So by Noether formula $b_2=-K^2=0$. Hence, by Bogomolov theorem~\cite{Teleman1994,LiYauZheng1994}, it is Hopf or Inoue. However, Inoue surfaces contain no curves and Hopf surfaces also cannot have a reduced, connected, smooth anticanonical divisor. So $M_a$ has to be K\"ahler. The classification now implies that $M_a$ is rational or irrational ruled. In the latter case, let $f:M_a\to B$ be the ruling over a curve of positive genus. Since $E_a\cdot F=2$ for a general fiber $F$, the map $f|_{E_a}:E_a\to B$ is nonconstant, i.e. it has positive degree. Then pullback gives an injective, hence nonzero map
$$
    H^1(B,\mathcal O_B)\longrightarrow H^1(M_a,\mathcal O_{M_a})\longrightarrow H^1(E_a,\mathcal O_{E_a}).
$$
However, the exact sequence
$$
    0\longrightarrow K_{M_a}\longrightarrow\mathcal O_{M_a}
    \longrightarrow\mathcal O_{E_a}\longrightarrow 0
$$
shows that the restriction map
$$
    H^1(M_a,\mathcal O_{M_a})\longrightarrow H^1(E_a,\mathcal O_{E_a})
$$
is zero. Thus, $M_a$ has to be rational.

The following discussion refers to the proof of \cite[Proposition 3.1.4]{AR24} (Although this Proposition needs additional hypotheses, the following part is also true in our setting). Every $(-1)$-curve $C\subset M_a$ satisfies $E_a\cdot C=1$, so contracting $C$ preserves the smoothness and anticanonical property of $E_a$. The minimal rational model is either $\mathbb P^2$ or a Hirzebruch surface $\mathbb F_n$. In the latter case, intersection with the negative section gives $2-n\geq 0$, so the minimal model is $\mathbb F_0$ or $\mathbb F_2$. It is easy to check blowing up one point on $\mathbb F_0$ or blowing up one point on $\mathbb{F}_2$ away from the negative section is the same as successive blowing up two points on $\mathbb P^2$. Consequently,
$$
    M_a\simeq \operatorname{Bl}_{p_1,\ldots,p_9}\mathbb P^2.
$$
Since the anticanonical divisor of $\mathbb{P}^2$ is an elliptic curve, we know that the nine blow-up centers lie on its successive strict transforms.
\end{proof}
\begin{remark}
    Every hyperk\"ahler ALH gravitational instanton, after hyperk\"ahler rotation, is recovered by Theorem~\ref{thm:ma} from a K\"ahler class on a compactification obtained by nine successive blow-ups of $\PP^2$ along a smooth cubic and its successive strict transforms. We just need to choose complex structure of model in Section~\ref{sec:coordinate} as in Remark~\ref{rmk:split-metric} such that it is a special case of the cylindrical metrics in Section~\ref{sec:cplx-model} and use the reconstruction result in Theorem~\ref{thm:global_compactification}.
\end{remark}


\begin{thebibliography}{99}

\bibitem{AR24}
Anna Abasheva, Rodion D\'eev, ``Complex Surfaces With Many Algebraic Structures.'' International Mathematics Research Notices, Volume 2024, Issue 9, May 2024, Pages 7379–7400, https://doi.org/10.1093/imrn/rnad190.

\bibitem{BHPV04}
Barth, Wolf P. and Hulek, Klaus and Peters, Chris A. M. and Van de Ven, Antonius. ``Compact Complex Surfaces.'' 2nd ed., Springer Berlin / Heidelberg, 2014.

\bibitem{ChenChenI}
Chen, Gao, and Xiuxiong Chen. ``Gravitational instantons with faster than quadratic curvature decay. I.'' Acta Math. \textbf{227} (2021), no.~2, 263--307.

\bibitem{ChenChenIII}
Chen, Gao, and Xiuxiong Chen. ``Gravitational instantons with faster than quadratic curvature decay (III)'' Mathematische Annalen 380.1-2 (2021): 687--717.

\bibitem{CH13}
Conlon, Ronan J., and Hans-Joachim Hein. ``Asymptotically conical Calabi--Yau manifolds, I.'' Duke Mathematical Journal 162.15 (2013): 2855--2902.

\bibitem{ConlonHein15}
Conlon, Ronan J., and Hans-Joachim Hein. ``Asymptotically conical Calabi--Yau metrics on quasi-projective varieties.'' Geometric and Functional Analysis 25.2 (2015): 517--552.

\bibitem{FT26}
Filip, Simion, and Valentino Tosatti. ``Nonlinearizable embeddings of elliptic curves in rational surfaces.'' arXiv preprint arXiv:2605.04160 (2026).

\bibitem{GuedjZeriahi2017}
Guedj, Vincent, and Ahmed Zeriahi. ``Degenerate complex Monge--Amp\`ere equations.'' EMS Tracts in Mathematics, vol.~26, European Mathematical Society, Z\"urich, 2017.

\bibitem{HHN15}
Haskins, Mark, Hans-Joachim Hein, and Johannes Nordstr\"om. ``Asymptotically cylindrical Calabi--Yau manifolds.'' Journal of Differential Geometry 101.2 (2015): 213--265.

\bibitem{Hein12}
Hein, Hans-Joachim. ``Gravitational instantons from rational elliptic surfaces.'' Journal of the American Mathematical Society 25.2 (2012): 355--393.

\bibitem{HSVZ2022}
Hein, Hans-Joachim, Song Sun, Jeff Viaclovsky, and Ruobing Zhang. ``Nilpotent structures and collapsing Ricci flat metrics on the K3 surface.'' Journal of the American Mathematical Society 35.1 (2022): 123--209.

\bibitem{LiYauZheng1994}
Li, Jun, Shing-Tung Yau, and Fangyang Zheng. ``On projectively flat Hermitian manifolds.'' Communications in Analysis and Geometry 2.1 (1994): 103--109.


\bibitem{Teleman1994}
Teleman, Andrei-Dumitru. ``Projectively flat surfaces and Bogomolov's theorem on class $VII_0$ surfaces.'' International Journal of Mathematics 5.02 (1994): 253--264.

\bibitem{TY}
Tian, Gang, and Shing-Tung Yau. ``Complete K\"ahler manifolds with zero Ricci curvature. I.'' Journal of the American Mathematical Society. \textbf{3} (1990), no.~3, 579--609.

\bibitem{JZ24}
Zhang, Junsheng. ``Hermitian--Yang--Mills connections on some complete non-compact K\"ahler manifolds.'' Mathematische Annalen 390.3 (2024): 4535--4575.

\end{thebibliography}
\end{document}